\documentclass[11pt]{amsart}

\usepackage[T1]{fontenc}
\usepackage{lmodern}
\usepackage{microtype}
\usepackage{amssymb,mathtools}
\usepackage{xcolor}

\definecolor{citationblue}{RGB}{0,68,136}
\usepackage[
  colorlinks=true,
  linkcolor=citationblue,
  citecolor=citationblue,
  urlcolor=citationblue
]{hyperref}

\hypersetup{
  pdftitle={Thirteen Results in Magnitude Theory},
  pdfauthor={Venkata Siddharth Pendyala}
}

\newtheorem{theorem}{Theorem}[section]
\newtheorem*{roffWillertonConjecture}
  {Conjecture 6.2 (Roff--Willerton)}
\newtheorem{corollary}[theorem]{Corollary}

\newcommand{\M}{\mathcal{M}}
\newcommand{\aff}{\operatorname{aff}}
\newcommand{\Mag}{\operatorname{Mag}}

\title[Thirteen Results in Magnitude Theory]
  {Thirteen Results in Magnitude Theory}
\author[V. S. Pendyala]{Venkata Siddharth Pendyala}
\email{venkatasiddharthpendyala@gmail.com}

\begin{document}

\begin{abstract}
We prove thirteen results concerning three aspects of magnitude theory.
The first four concern small-scale magnitude and microscopic
weightings.  We identify the derivative of the magnitude function at
the origin with maximal energy, allowing the value \(+\infty\), and
construct finite spaces for which small-scale magnitude diverges or
for which the one-point property and finite concentration do not
produce a microscopic weighting.  A fifth result disproves the
conjectured monotonicity of magnitude by giving a fourteen-point metric
space of strict negative type whose magnitude function decreases.
The remaining eight results concern continuity and dimension.
Magnitude is Hausdorff-continuous at finite subsets of
finite-dimensional positive definite normed spaces, but need not be
continuous on compact subsets, even under convexity, geodesicity,
bounded tractability, or simultaneous Hausdorff and Lipschitz
convergence.  The same constructions disprove an
\(\ell_1\)-intrinsic-volume formula and two tractability conjectures.
Finally, we construct a countable compact space of negative type whose
Minkowski dimension is zero and whose lower and upper magnitude
dimensions are both infinite.
\end{abstract}

\maketitle

\section{Introduction}

Magnitude is a scale-sensitive notion of the size of a metric space.
It was introduced by Leinster through an analogy between the Euler
characteristic of a finite category and an invariant of a finite metric
space regarded as an enriched category \cite{Leinster}.  For a finite
space, magnitude is calculated from the similarity kernel
\(e^{-d(x,y)}\); heuristically, nearby points are counted as partially
overlapping, while sufficiently separated points make essentially
independent contributions.  Rescaling the metric thus gives a magnitude
function \(t\mapsto |tX|\) that views the same space at different
resolutions.  In the positive definite setting, magnitude also has
variational and capacity formulations
\cite{MeckesPositive,MeckesCapacity}, placing it at an intersection of
category theory, metric geometry, potential theory, and geometric
measure theory \cite{LeinsterMeckes}.

This paper proves thirteen results in magnitude theory, divided among
small-scale behavior, monotonicity, and continuity.  The first four
concern the magnitude function near the origin and the existence of
microscopic weightings.  We prove the Roff--Willerton conjecture \cite[Conjecture~6.2]{RoffWillerton}
identifying the right derivative at zero with maximal energy, in fact
for all compact spaces of negative type and with the value \(+\infty\)
allowed.  We then answer a question of Roff and Yoshinaga \cite[p.~4, Question~(3)]{RoffYoshinaga} by
constructing a finite space whose magnitude tends to \(+\infty\) at
small scales.  Two further finite examples show, respectively, that
finite concentration and the one-point property do not imply the
existence of a microscopic weighting.

The fifth result concerns monotonicity.  Gomi and Meckes asked whether
magnitude is monotone under enlargement of the metric for stably
positive definite spaces, and hence whether the magnitude function of
every space of negative type is increasing \cite[paragraph following Theorem~1.4]{GomiMeckes}.  We give an explicit
fourteen-point metric space of strict negative type whose magnitude
function has negative derivative at some scale.  This complements the
semimetric counterexamples of Gomi and Tsuji \cite[Introduction]{GomiTsuji}, who recorded that metric
counterexamples had been communicated by the present author.

The remaining eight results concern continuity and dimension.  We
prove the Katsumasa--Roff--Yoshinaga continuity conjecture \cite[Conjecture~1.3]{KatsumasaRoffYoshinaga} for finite
subsets of finite-dimensional positive definite normed spaces.  For
compact subsets the situation is different: magnitude can fail to be
Hausdorff-continuous even for convex, and hence geodesic, subsets of
\(\ell_1^3\).  This example also disproves the Leinster--Meckes
\(\ell_1\)-intrinsic-volume conjecture \cite[Conjecture~4.5]{LeinsterMeckes} and yields discontinuity inside
a single compact ambient space and inside a bounded tractable space.
It consequently disproves the compact-subspace part of two conjectures
of Kališnik and Lešnik \cite[Conjectures~5.1 and~5.2]{KalisnikLesnikTractable}.  A separate construction shows that magnitude
may diverge under simultaneous Lipschitz and Hausdorff convergence.
The final example has Minkowski dimension zero but infinite lower and
upper magnitude dimensions.

The common background is classical.  Schoenberg realized metrics of
negative type as squared Hilbert-space distances
\cite{Schoenberg1935,Schoenberg1938}.  Nickolas and Wolf related the
maximal energy of a compact metric space to the radius of a spherical
Schoenberg realization \cite{NickolasWolf}.  Meckes developed magnitude
for positive definite compact spaces and proved its capacity
description \cite{MeckesPositive,MeckesCapacity}.  At the finite level,
Asao and Gomi interpreted magnitude by means of circumspheres
\cite{AsaoGomi}, while Devriendt obtained a Euclidean circumradius model
and used it to study subspaces \cite{Devriendt}.  These results are used
as tools and are cited where they enter.
\section{Preliminaries}

Throughout, \((X,d)\) is a compact metric space, and all Hilbert spaces
are real. Let \(\M(X)\) denote the vector space of finite signed Borel
measures on \(X\). For \(\nu\in\M(X)\), its total mass is \(\nu(X)\)
and its total variation norm is \(\|\nu\|_{\mathrm{TV}}\). We write
\(\delta_x\) for the unit point mass at \(x\).

The \emph{energy} of \(\nu,\eta\in\M(X)\) is
\[
 I(\nu,\eta)
 =\iint_{X\times X}d(x,y)\,d\nu(x)\,d\eta(y),
 \qquad
 I(\nu)=I(\nu,\nu).
\]
These integrals are finite because \(d\) is bounded. The
\emph{maximal energy} of \(X\) is
\begin{equation}\label{eq:maximal-energy}
 M(X)=\sup\{I(\nu):\nu\in\M(X),\ \nu(X)=1\}.
\end{equation}
The supremum is taken over signed measures, not merely probability
measures, and may be infinite.

The space \(X\) is of \emph{negative type} if, for every finite choice
of points \(x_1,\ldots,x_n\in X\) and real numbers
\(a_1,\ldots,a_n\) with \(\sum_i a_i=0\),
\begin{equation}\label{eq:negative-type}
 \sum_{i,j=1}^n a_i a_jd(x_i,x_j)\leq0.
\end{equation}
It is of \emph{strict negative type} if the inequality is strict
whenever the points are distinct and the coefficients are not all
zero.

A \emph{Schoenberg realization} is a continuous map \(v:X\to H\)
into a real Hilbert space such that
\[
 d(x,y)=\|v_x-v_y\|^2.
\]
It is \emph{spherical} if \(\|v_x-c\|\) is independent of \(x\) for
some \(c\in H\), and \emph{centered} if
\[
 0\in\overline{\aff}\,v(X).
\]
Here \(\aff A\) is the algebraic affine hull of \(A\), and the bar
denotes norm closure.

For \(t\geq0\), let \(tX\) denote \(X\) with metric \(td\), and put
\[
 \mathcal E_t(\nu,\eta)
 =\iint_{X\times X}e^{-td(x,y)}\,d\nu(x)\,d\eta(y),
 \qquad
 \mathcal E_t(\nu)=\mathcal E_t(\nu,\nu).
\]
A real symmetric kernel \(K\) is \emph{positive definite} if
\[
 \sum_{i,j=1}^n a_i a_jK(x_i,x_j)\geq0
\]
for every finite choice of points and real coefficients. A metric
space \(X\) is \emph{positive definite} if, for every finite set of
distinct points \(x_1,\ldots,x_n\in X\), the matrix
\[
 \bigl(e^{-d(x_i,x_j)}\bigr)_{i,j=1}^n
\]
is positive definite. A normed space is positive definite when its
induced metric is positive definite. Schoenberg's theorem identifies
negative type with the positive definiteness of \(e^{-td}\) for every
\(t>0\) \cite{Schoenberg1938}.

For a finite metric space \(F=\{x_1,\ldots,x_n\}\), write
\[
 Z_F(t)=\bigl(e^{-td(x_i,x_j)}\bigr)_{i,j=1}^n
\]
for the \emph{similarity matrix} of \(tF\), and let
\(\mathbf1\in\mathbb R^n\) denote the all-ones column vector. A
\emph{weighting} on \(tF\) is a vector \(w\in\mathbb R^n\) satisfying
\[
 Z_F(t)w=\mathbf1.
\]
When \(Z_F(t)\) is invertible, the weighting is unique and is given by
\[
 w_F(t)=Z_F(t)^{-1}\mathbf1.
\]
Its coordinate sum is the magnitude:
\[
 |tF|
 =\Mag(tF)
 =\mathbf1^Tw_F(t)
 =\mathbf1^TZ_F(t)^{-1}\mathbf1.
\]
More generally, whenever a weighting exists, its coordinate sum is
independent of the chosen weighting and defines \(|tF|\).

For a compact space of negative type, we use Meckes's capacity formula
\begin{equation}\label{eq:magnitude-capacity}
 \Mag(tX)
 =
 \left(
  \inf_{\substack{\nu\in\M(X)\\ \nu(X)=1}}
  \mathcal E_t(\nu)
 \right)^{-1},
\end{equation}
with \(1/0=+\infty\)
\cite{MeckesPositive,MeckesCapacity}. We use
\[
 |tX|=\Mag(tX)
\]
interchangeably, and in particular write \(|X|=\Mag(X)\). At \(t=0\)
the infimum in \eqref{eq:magnitude-capacity} equals one, and we put
\[
 |0X|=\Mag(0X)=1.
\]

Roff and Willerton use the term \emph{magnitude function} for
\cite[Definition~6.1]{RoffWillerton}
\begin{equation}\label{eq:RW-magnitude-function}
 \mu_X(t)
 =
 \sup\{\Mag(tF):F\subseteq X\text{ is finite}\},
 \qquad t>0.
\end{equation}
For compact spaces of negative type this agrees with
\eqref{eq:magnitude-capacity}. We set \(\mu_X(0)=1\); throughout,
\(\mu_X'(0)\) denotes the extended right derivative
\(\lim_{t\downarrow0}(\mu_X(t)-1)/t\), which may equal \(+\infty\),
and no continuity at zero is assumed. A metric space \(X\)
has the \emph{one-point property} if
\[
 \lim_{t\downarrow0}|tX|=1.
\]

For a real Hilbert space \(H\), let
\[
 \mathcal F_s(H)
 =
 \bigoplus_{m=0}^{\infty}H^{\odot m}
\]
denote its symmetric Fock space, where \(H^{\odot m}\) is the
\(m\)-fold symmetric Hilbert tensor power and
\(H^{\odot0}=\mathbb R\). Its \emph{vacuum vector} is
\[
 \Omega=(1,0,0,\ldots).
\]

For pointed metric spaces \((X,x_0)\) and \((Y,y_0)\), their
\emph{metric wedge} \(X\vee Y\) is obtained by identifying \(x_0\)
with \(y_0\), retaining the original metrics on the two branches, and
putting
\[
 d(x,y)=d_X(x,x_0)+d_Y(y_0,y)
 \qquad(x\in X,\ y\in Y).
\]
The same definition applies to an arbitrary family of pointed spaces,
whose wedge is denoted by \(\bigvee_iX_i\). Their
\(\ell_1\)-product \(X\times_1Y\) is the Cartesian product equipped
with
\[
 d\bigl((x,y),(x',y')\bigr)
 =
 d_X(x,x')+d_Y(y,y').
\]
We write \(X^{\times_1m}\) for the \(m\)-fold \(\ell_1\)-product.

A subset \(K\subseteq\ell_1^n\) is \emph{\(\ell_1\)-convex} if it is
geodesic in its induced \(\ell_1\)-metric: every two points of \(K\)
are joined in \(K\) by a geodesic of length equal to their
\(\ell_1\)-distance. In particular, every ordinarily convex subset of
\(\ell_1^n\) is \(\ell_1\)-convex.

For a nonempty compact convex set \(K\subseteq\ell_1^n\), its
\emph{\(\ell_1\)-intrinsic volumes} are defined by
\[
 V_0'(K)=1,
 \qquad
 V_i'(K)
 =
 \sum_{\substack{I\subseteq\{1,\ldots,n\}\\ |I|=i}}
 \operatorname{vol}_i(\pi_IK),
 \qquad 1\leq i\leq n,
\]
where \(\pi_I:\mathbb R^n\to\mathbb R^I\) is the corresponding
coordinate projection and \(\operatorname{vol}_i\) is
\(i\)-dimensional Lebesgue measure
\cite[Section~4]{LeinsterMeckes}.

\section{Small-scale magnitude and microscopic weightings}

Roff and Willerton state the following conjecture
\cite[Conjecture~6.2]{RoffWillerton}.

\begin{roffWillertonConjecture}
For every compact metric space $Y$ of strictly negative type, we have
\[
  \mu_Y'(0)=M(Y).
\]
\end{roffWillertonConjecture}

\begin{theorem}[First variation of magnitude]\label{thm:first-variation}
For every compact metric space $Y$ of negative type,
\[
  \mu_Y'(0)=M(Y),
\]
where the value $+\infty$ is allowed.  In particular,
Conjecture~6.2 is true.
\end{theorem}

\begin{proof}
Suppose first that $M(Y)<\infty$.  The spherical normalization used
below is due to Nickolas and Wolf \cite[Theorem~4.8]{NickolasWolf}.
It gives a real Hilbert space $H$ and a centered spherical Schoenberg
realization $v:Y\to H$ such that
\begin{equation}\label{eq:centered-sphere}
  d(y,z)=\|v_y-v_z\|^2,
  \qquad
  \|v_y\|^2=\frac{M(Y)}2.
\end{equation}
Since all the points $v_y$ lie on the same sphere,
\begin{equation}\label{eq:spherical-distance}
  d(y,z)=M(Y)-2\langle v_y,v_z\rangle.
\end{equation}

Let $\nu\in\M(Y)$ have mass one, and write
\[
  b(\nu)=\int_Yv_y\,d\nu(y).
\]
The integral exists because $v$ is continuous and $Y$ is compact.
Integrating \eqref{eq:spherical-distance} twice gives the elementary
identity
\begin{equation}\label{eq:energy-center}
  I(\nu)=M(Y)-2\|b(\nu)\|^2.
\end{equation}
This identity contains the first-order part of the argument.

We next obtain an estimate which is uniform in $\nu$.  For $t\geq0$ set
\[
  L_t(y,z)=e^{2t\langle v_y,v_z\rangle}-1.
\]
The kernel $L_t$ is positive definite.  Indeed, the matrix
$(\langle v_{y_i},v_{y_j}\rangle)_{i,j}$ is positive semidefinite for
every finite set of points.  Entrywise products of positive
semidefinite matrices are positive semidefinite, and the power series
\[
  L_t(y,z)=\sum_{k=1}^{\infty}
     \frac{(2t)^k}{k!}\langle v_y,v_z\rangle^k
\]
has nonnegative coefficients.  It follows first for finitely supported
signed measures.  To pass to an arbitrary $\nu\in\M(Y)$, choose finite
Borel partitions of $Y$ whose mesh tends to zero and replace the
restriction of $\nu$ to each part by a point mass having the same
mass.  The resulting atomic measures have total variation at most
$\|\nu\|_{\mathrm{TV}}$, and the uniform continuity of $L_t$ on the
compact space $Y\times Y$ gives convergence of the corresponding
double integrals.  Hence
\begin{equation}\label{eq:positive-remainder}
  \iint_{Y\times Y}L_t(y,z)\,d\nu(y)\,d\nu(z)\geq0.
\end{equation}

If $\nu(Y)=1$, equations \eqref{eq:spherical-distance} and
\eqref{eq:positive-remainder} give
\begin{align}
  \mathcal E_t(\nu)
  &=e^{-tM(Y)}
    \iint_{Y\times Y}e^{2t\langle v_y,v_z\rangle}
       \,d\nu(y)\,d\nu(z) \notag\\
  &\geq e^{-tM(Y)}.                         \label{eq:energy-lower-bound}
\end{align}
Taking the infimum in \eqref{eq:magnitude-capacity}, we obtain
\begin{equation}\label{eq:magnitude-upper-bound}
  \mu_Y(t)\leq e^{tM(Y)}.
\end{equation}
Consequently,
\begin{equation}\label{eq:derivative-upper-bound}
  \limsup_{t\downarrow0}
  \frac{\mu_Y(t)-1}{t}\leq M(Y).
\end{equation}

It remains to prove the opposite inequality.  Since the realization is
centered, for every $\varepsilon>0$ there are points
$y_1,\ldots,y_n\in Y$ and real numbers $a_1,\ldots,a_n$ such that
\[
  \sum_{i=1}^n a_i=1,
  \qquad
  \left\|\sum_{i=1}^na_iv_{y_i}\right\|<\varepsilon.
\]
For the signed measure
$\nu_\varepsilon=\sum_i a_i\delta_{y_i}$,
\eqref{eq:energy-center} yields
\begin{equation}\label{eq:almost-maximal-energy}
  I(\nu_\varepsilon)>M(Y)-2\varepsilon^2.
\end{equation}
The measure $\nu_\varepsilon$ is now fixed.  Since $d$ is bounded, the
Taylor expansion of $e^{-td(y,z)}$ is uniform on $Y\times Y$; hence
\begin{equation}\label{eq:fixed-measure-expansion}
  \mathcal E_t(\nu_\varepsilon)
  =1-tI(\nu_\varepsilon)+O_{\nu_\varepsilon}(t^2).
\end{equation}
The capacity formula gives
$\mu_Y(t)^{-1}\leq\mathcal E_t(\nu_\varepsilon)$.  Taking reciprocals
in \eqref{eq:fixed-measure-expansion}, for sufficiently small $t>0$,
we find
\[
  \mu_Y(t)
  \geq1+tI(\nu_\varepsilon)+O_{\nu_\varepsilon}(t^2).
\]
Together with \eqref{eq:almost-maximal-energy}, this implies
\[
  \liminf_{t\downarrow0}
  \frac{\mu_Y(t)-1}{t}
  \geq M(Y)-2\varepsilon^2.
\]
Letting $\varepsilon\downarrow0$ and comparing with
\eqref{eq:derivative-upper-bound} proves the theorem when $M(Y)$ is
finite.

Suppose finally that $M(Y)=+\infty$.  Given $A>0$, the definition of
maximal energy supplies a mass-one measure $\nu$ such that $I(\nu)>A$.
For this fixed measure the same uniform Taylor expansion gives
\[
  \mathcal E_t(\nu)=1-tI(\nu)+O_\nu(t^2).
\]
Since $\mu_Y(t)^{-1}\leq\mathcal E_t(\nu)$, it follows that
\[
  \liminf_{t\downarrow0}
  \frac{\mu_Y(t)-1}{t}\geq I(\nu)>A.
\]
The number $A$ is arbitrary, so the right derivative is $+\infty$.
\end{proof}

Roff and Yoshinaga ask
\cite[p.~4, Question~(3)]{RoffYoshinaga}:
\begin{quote}
\emph{Does there exist a finite metric space \(X\) with
\(\lim_{t\to0}|tX|=\infty\)?}
\end{quote}

Kamiyama introduced the cyclic two-chunk family used below, computed
its magnitude, and analyzed its finite small-scale limits
\cite[Proposition~3.2 and equations~(5.49)--(5.52)]{Kamiyama}.
His calculation identifies a critical parameter at which the
denominator in the resulting second-order small-scale formula
vanishes.  Kamiyama records the behavior as the parameter approaches
this critical value.  We instead fix the parameter at that value and
determine the behavior as the scale tends to zero.

\begin{theorem}[Infinite small-scale magnitude]
\label{thm:infinite-small-scale}
There exists a twelve-point finite metric space \(X\) such that
\[
    \lim_{t\downarrow0}|tX|=+\infty.
\]
Consequently, Question~(3) of Roff and Yoshinaga has an affirmative
answer.
\end{theorem}

\begin{proof}
For \(1\leq r\leq5\), consider Kamiyama's twelve-point cyclic
two-chunk space
\[
    X_r=\{a_i,b_i:i\in\mathbb Z/6\mathbb Z\},
\]
with subscripts read modulo \(6\), whose nonzero distances are
\[
    d(a_i,a_j)=r,\qquad d(b_i,b_j)=1
    \quad (i\neq j),
\]
and
\[
    d(a_i,b_j)=
    \begin{cases}
        (r+5)/4,&j-i\in\{0,1\},\\[2mm]
        r/2,&j-i\notin\{0,1\}.
    \end{cases}
\]
Kamiyama verifies that these parameters define a metric for
\(1\leq r\leq5\).  For completeness, if
\[
    \gamma=\frac{r+5}{4},
    \qquad
    \delta=\frac r2,
\]
then
\[
    |\gamma-\delta|
      =\frac{5-r}{4}\leq\min\{1,r\},
    \qquad
    2\delta=r\geq1,
\]
while
\[
    2\gamma=\frac{r+5}{2}\geq r,
    \qquad
    \gamma+\delta=\frac{3r+5}{4}\geq r.
\]
These are precisely the nontrivial mixed triangle inequalities.

Set
\[
    A_r(t)=1+5e^{-rt},\qquad
    B(t)=1+5e^{-t},
\]
and
\[
    C_r(t)=2e^{-(r+5)t/4}+4e^{-rt/2}.
\]
Kamiyama's magnitude formula specializes to
\begin{equation}\label{eq:cyclic-magnitude}
    |tX_r|=\frac{N_r(t)}{E_r(t)},
\end{equation}
where
\[
    N_r(t)=6\bigl(A_r(t)+B(t)-2C_r(t)\bigr),
    \qquad
    E_r(t)=A_r(t)B(t)-C_r(t)^2.
\]

Now take the critical value
\[
    r_0=\frac{\sqrt{105}-5}{4}.
\]
At this value, the quadratic coefficient of \(E_{r_0}\) vanishes.
Expansion at \(t=0\), carried one order further in the denominator,
gives
\[
    N_{r_0}(t)
      =\frac{75}{8}(4-3r_0)t^2+O(t^3)
\]
and
\[
    E_{r_0}(t)
      =\frac{5}{32}(29r_0-22)t^3+O(t^4).
\]
Since \(1<r_0<4/3\), both leading coefficients are positive.  In
particular,
\[
    E_{r_0}(t)>0
\]
for all sufficiently small \(t>0\), so
\eqref{eq:cyclic-magnitude} applies there.  Dividing the two
expansions yields
\begin{equation}\label{eq:critical-asymptotic}
    |tX_{r_0}|
      =
      60\,\frac{4-3r_0}{29r_0-22}\,\frac1t+O(1).
\end{equation}
The coefficient of \(t^{-1}\) is positive.  Therefore
\[
    \lim_{t\downarrow0}|tX_{r_0}|=+\infty.
\]
Taking \(X=X_{r_0}\) proves the theorem.
\end{proof}

For a finite metric space \(X=\{x_1,\ldots,x_n\}\), let
\(D_X=(d(x_i,x_j))_{i,j}\) be its distance matrix.  Following Roff
and Willerton, a \emph{gauging} for \(D_X\) is a vector
\(g\in\mathbb R^n\) such that
\[
    \mathbf 1^Tg=1,
    \qquad
    D_Xg=c\mathbf 1
\]
for some \(c\in\mathbb R\).  In this case \(D_X\) has
\emph{finite concentration}.  If
\[
    w_X(t)=Z_X(t)^{-1}\mathbf 1
\]
for sufficiently small \(t>0\), then a \emph{microscopic weighting}
on \(X\) is the componentwise limit
\[
    \widehat w_X=\lim_{t\downarrow0}w_X(t),
\]
provided that this limit exists
\cite[Definitions~1.1, 2.3, and~2.4]{RoffWillerton}.

Roff and Willerton conjecture
\cite[Conjecture~3.3]{RoffWillerton}:
\begin{quote}
\emph{A finite metric space \(X\) has a microscopic weighting if and
only if its distance matrix satisfies
\(\operatorname{con}(D)\neq\infty\).}
\end{quote}

\begin{theorem}[Finite concentration without a microscopic weighting]
\label{thm:finite-concentration-counterexample}
There exists a twenty-four-point metric space \(Y\) whose distance
matrix has finite concentration but which does not admit a microscopic
weighting.  Consequently, Conjecture~3.3 of Roff and Willerton is
false.
\end{theorem}

\begin{proof}
Write \(X(\alpha,\beta,\gamma,\delta)\) for Kamiyama's twelve-point
\((6,2)\) cyclic two-chunk space: its points are
\[
    \{a_i,b_i:i\in\mathbb Z/6\mathbb Z\},
\]
the distances between distinct \(a\)-points and distinct \(b\)-points
are \(\alpha\) and \(\beta\), respectively, and
\[
    d(a_i,b_j)=
    \begin{cases}
        \gamma,&j-i\in\{0,1\},\\
        \delta,&j-i\notin\{0,1\}.
    \end{cases}
\]
Kamiyama's magnitude calculation gives
\begin{equation}\label{eq:two-chunk-orbit-formula}
    |tX|
      =6\,\frac{A(t)+B(t)-2C(t)}
                    {A(t)B(t)-C(t)^2},
\end{equation}
where
\[
    A(t)=1+5e^{-\alpha t},\qquad
    B(t)=1+5e^{-\beta t},\qquad
    C(t)=2e^{-\gamma t}+4e^{-\delta t}.
\]
The common weights on the two blocks are
\begin{equation}\label{eq:two-chunk-block-weights}
    x(t)=\frac{B(t)-C(t)}{A(t)B(t)-C(t)^2},
    \qquad
    y(t)=\frac{A(t)-C(t)}{A(t)B(t)-C(t)^2}.
\end{equation}

Set
\[
    r=\frac{5+4\sqrt5}{11},
\]
and consider the two metric spaces
\[
    P=X\left(r,1,\frac{r+5}{4},\frac r2\right),
    \qquad
    Q=X\left(3,1,\frac53,\frac53\right).
\]
These are members of the two metric families analyzed by Kamiyama
\cite[Proposition~5.2]{Kamiyama}.  Substitution into
\eqref{eq:two-chunk-orbit-formula} and
\eqref{eq:two-chunk-block-weights}, followed by expansion at \(t=0\),
gives
\begin{equation}\label{eq:PQ-small-scale-data}
    |tP|=O(t),\qquad
    |tQ|=2+O(t),
\end{equation}
while the common weight on the \(a\)-block of \(P\) satisfies
\begin{equation}\label{eq:P-divergent-weight}
    x_P(t)
      =-\frac{2}{5(r-1)}\,t^{-1}+O(1).
\end{equation}
For clarity, the first expansion follows from
\[
    11r^2-10r-5=0,
\]
whereas
\[
    A_P(t)B_P(t)-C_P(t)^2
      =-\frac{25}{4}(r-1)^2t^2+O(t^3).
\]
Thus the denominator used above is nonzero for all sufficiently small
\(t>0\).

Form the metric wedge \(W=P\vee Q\), identifying \(a_0\in P\) with
one point of \(Q\).  Weightings on the two branches combine at the
identified point, and hence
\[
    |tW|=|tP|+|tQ|-1.
\]
The quotients in \eqref{eq:two-chunk-orbit-formula} have removable
singularities at \(t=0\), so for some \(a\in\mathbb R\),
\begin{equation}\label{eq:W-magnitude-expansion}
    |tW|=1+at+O(t^2).
\end{equation}
Moreover, the weight at any \(a_i\neq a_0\) is still \(x_P(t)\), and
therefore diverges by \eqref{eq:P-divergent-weight}.

Put
\[
    c=\frac{2}{5(r-1)}
\]
and define \(p\in\mathbb R^W\) by
\[
    p(a_i)=c,\qquad p(b_i)=-c,
\]
on the copy of \(P\), and \(p=0\) on \(W\setminus P\).  Then
\[
    \mathbf 1^Tp=0.
\]
If
\[
    \gamma=\frac{r+5}{4},
    \qquad
    \delta=\frac r2,
\]
then \(2\gamma+4\delta=5(r+1)/2\).  Consequently, for every point in
either block of \(P\),
\[
    (D_Wp)_i
      =\frac{5c}{2}(r-1)=1.
\]
For a point \(z\) in the other branch, wedge distance gives
\[
    \sum_{x\in P}d_W(z,x)p(x)
      =d_W(z,a_0)\sum_{x\in P}p(x)
       +\sum_{x\in P}d_P(a_0,x)p(x)
      =1.
\]
Hence
\begin{equation}\label{eq:W-gauging-vector}
    D_Wp=\mathbf 1.
\end{equation}

Choose
\[
    s>\max\{\operatorname{diam}(W)/2,|a|\}
\]
and adjoin a point \(*\) at distance \(s\) from every point of \(W\).
The resulting space
\[
    Y=W\cup\{*\}
\]
is metric and has \(23+1=24\) points.  With the coordinates ordered
as \(W,*\), set
\[
    g=(-sp,1).
\]
Equations \(\mathbf 1^Tp=0\) and \eqref{eq:W-gauging-vector} give
\[
    \mathbf 1^Tg=1,
    \qquad
    D_Yg=0.
\]
Thus \(g\) is a gauging for \(D_Y\), of concentration \(0\).  In
particular, \(D_Y\) has finite concentration.

It remains to show that \(Y\) has no microscopic weighting.  Put
\[
    q=e^{-st},\qquad m(t)=|tW|,
\]
and let \(w_W(t)\) be the weighting on \(tW\).  The similarity matrix
of \(tY\) has block form
\[
    Z_Y(t)=
    \begin{pmatrix}
        Z_W(t)&q\mathbf 1\\
        q\mathbf 1^T&1
    \end{pmatrix}.
\]
Solving its weighting equations shows that the restriction
\(u_W(t)\) of the weighting of \(tY\) to \(W\) is
\begin{equation}\label{eq:Y-weighting-restriction}
    u_W(t)
      =\frac{1-q}{1-q^2m(t)}\,w_W(t).
\end{equation}
By \eqref{eq:W-magnitude-expansion},
\[
    1-q=st+O(t^2),
    \qquad
    1-q^2m(t)=(2s-a)t+O(t^2).
\]
Since \(s>|a|\),
\[
    \lim_{t\downarrow0}
      \frac{1-q}{1-q^2m(t)}
      =\frac{s}{2s-a}\neq0.
\]
The coordinate of \(w_W(t)\) at any \(a_i\neq a_0\) diverges by
\eqref{eq:P-divergent-weight}; therefore the corresponding coordinate
of \(u_W(t)\) also diverges.  The weighting of \(tY\) has no
componentwise limit as \(t\downarrow0\), so \(Y\) has no microscopic
weighting.
\end{proof}

Roff and Willerton observe
\cite[Remark~3.4]{RoffWillerton}:
\begin{quote}
\emph{Part~(1) of Theorem~3.1 says that the existence of a microscopic
weighting on \(X\) implies the one-point property, but we do not know
whether the converse is true.  It is conceivable that some
cancellation could cause
\(\mu_X(t)=\mathbf 1^Tw(t)\) to converge to \(1\) as \(t\to0\), even
as the entries of \(w(t)\) diverge.}
\end{quote}

\begin{theorem}[The one-point property without a microscopic weighting]
\label{thm:one-point-no-microscopic-weighting}
There exists a twenty-three-point metric space \(W\) such that
\[
    \lim_{t\downarrow0}|tW|=1,
\]
but \(W\) does not admit a microscopic weighting.  Thus the converse
considered by Roff and Willerton is false.
\end{theorem}

\begin{proof}
Let \(W=P\vee Q\) be the twenty-three-point wedge constructed above.
We have
\[
    |tW|=1+at+O(t^2),
\]
and therefore
\[
    \lim_{t\downarrow0}|tW|=1.
\]
Thus \(W\) has the one-point property.  On the other hand, for every
\(a_i\neq a_0\), the corresponding coordinate of its weighting is
\[
    w_W(t)(a_i)
      =-\frac{2}{5(r-1)}\,t^{-1}+O(1),
\]
which diverges as \(t\downarrow0\).  Hence \(w_W(t)\) has no
componentwise limit, so \(W\) admits no microscopic weighting.
\end{proof}

\section{Nonmonotonicity}

Leinster and Meckes conjectured that the magnitude function
\[
        t\longmapsto |tX|
\]
is increasing for every compact metric space \(X\) of negative type
\cite[p.~29 of the arXiv version]{LeinsterMeckes}. Gomi and Meckes
later returned to this question and, after Theorem~1.4, considered the
more general assertion that magnitude is monotone under pointwise
enlargement of the metric in the stably positive definite setting
\cite[paragraph following Theorem~1.4]{GomiMeckes}. The original
conjecture is the special case obtained by rescaling a single metric.

Gomi and Tsuji subsequently constructed nonmonotone negative-type
semimetric spaces and recorded that genuine metric counterexamples had
been communicated by the present author
\cite[Introduction]{GomiTsuji}. The following gives an explicit
counterexample which is both a metric space and of strict negative
type.

\begin{theorem}[A nonmonotone magnitude function]
\label{thm:nonmonotone-magnitude}
There exist a fourteen-point metric space \(X\) of strict negative type
and \(t_0>0\) such that
\[
 \left.\frac{d}{dt}\Mag(tX)\right|_{t=t_0}<0.
\]
Consequently, the monotonicity conjecture of Leinster and Meckes, and
its comparison-metric extension discussed by Gomi and Meckes, are
false.
\end{theorem}

\begin{proof}
Let \(S=\{0,1,3\}\subset\mathbb Z/7\mathbb Z\), and let
\(P=(P_{ij})\) be the \(7\times7\) matrix defined by
\[
 P_{ij}=1
 \quad\Longleftrightarrow\quad
 j-i\in S.
\]
The six nonzero differences of distinct elements of \(S\) are precisely
the six nonzero residues modulo \(7\).  Hence, with \(J\) denoting the
all-ones matrix,
\[
 P\mathbf1=3\mathbf1,\qquad
 P^T\mathbf1=3\mathbf1,\qquad
 PP^T=2I+J.
\]
In particular, the operator norm of \(P\) on \(\mathbf1^\perp\) is
\(\sqrt2\).

For \(\varepsilon\geq0\), let
\[
 X_\varepsilon=\{a_0,\ldots,a_6,b_0,\ldots,b_6\}
\]
and, for distinct points, define
\[
 d(a_i,a_j)=\frac78,\qquad
 d(b_i,b_j)=\frac98,
\]
and
\[
 d(a_i,b_j)=
 \begin{cases}
  \dfrac{26}{21}+\varepsilon,&j-i\in S,\\[2mm]
  \dfrac47+\varepsilon,&j-i\notin S.
 \end{cases}
\]
These distances satisfy the triangle inequality.  Indeed, if
\(\delta=4/7+\varepsilon\) and
\(\gamma=26/21+\varepsilon\), every nontrivial case follows from
\[
 \frac98<2\delta,
 \qquad
 \gamma-\delta=\frac23<\frac78.
\]

Let \(D_\varepsilon\) be the distance matrix.  If
\(z=(u,v)\in\mathbb R^7\oplus\mathbb R^7\) satisfies
\(\mathbf1^Tu+\mathbf1^Tv=0\), write
\[
 u=s\mathbf1+u_0,\qquad
 v=-s\mathbf1+v_0,
 \qquad
 u_0,v_0\perp\mathbf1.
\]
Since
\[
 3\frac{26}{21}+4\frac47
 =
 6
 =
 3\left(\frac78+\frac98\right),
\]
a direct calculation gives
\begin{equation}\label{eq:nonmonotone-distance-form}
 z^TD_\varepsilon z
 =
 -98\varepsilon s^2
 -\frac78\|u_0\|^2
 -\frac98\|v_0\|^2
 +\frac43u_0^TPv_0.
\end{equation}
Now \(\|P\|_{\mathbf1^\perp}=\sqrt2\), and
\[
 \frac78\frac98-\left(\frac{2\sqrt2}{3}\right)^2
 =
 \frac{55}{576}>0.
\]
Thus the last three terms in
\eqref{eq:nonmonotone-distance-form} form a negative definite
quadratic form in \((u_0,v_0)\).  It follows that
\(z^TD_\varepsilon z<0\) for every nonzero \(z\perp\mathbf1\) whenever
\(\varepsilon>0\).  Hence \(X_\varepsilon\) has strict negative type.

It remains to prove nonmonotonicity.  At \(\varepsilon=0\), simultaneous
cyclic translation of the two blocks acts by isometries.  For the
similarity matrix
\[
 Z_0(t)=\bigl(e^{-td(x,y)}\bigr)_{x,y\in X_0},
\]
put
\[
 A=1+6e^{-7t/8},\qquad
 B=1+6e^{-9t/8},\qquad
 C=3e^{-26t/21}+4e^{-4t/7}.
\]
On the block-constant subspace, \(Z_0(t)\) is represented by
\[
 \begin{pmatrix}A&C\\ C&B\end{pmatrix}.
\]
On its invariant complement,
\[
 Z_0(t)=-tD_0+O(t^2),
\]
and \(-D_0\) is positive definite there by
\eqref{eq:nonmonotone-distance-form}.  Moreover,
\[
 AB-C^2=\frac{73}{96}t^2+O(t^3)>0
\]
for all sufficiently small \(t>0\).  Thus \(Z_0(t)\) is invertible
there, and symmetry gives
\begin{equation}\label{eq:nonmonotone-magnitude-expansion}
 \Mag(tX_0)
 =
 7\,\frac{A+B-2C}{AB-C^2}
 =
 \frac{127}{73}
 -\frac{12672}{37303}t
 +O(t^2).
\end{equation}
Therefore \(\frac{d}{dt}\Mag(tX_0)<0\) for all sufficiently small
\(t>0\).  Fix such a \(t_0\).  Since \(Z_0(t_0)\) is invertible,
\[
 \left.
 \frac{\partial}{\partial t}
 \bigl(\mathbf1^TZ_\varepsilon(t)^{-1}\mathbf1\bigr)
 \right|_{t=t_0}
\]
depends continuously on \(\varepsilon\) near \(0\).  It remains
negative for every sufficiently small \(\varepsilon>0\).  Choosing one
such \(\varepsilon\) and setting \(X=X_\varepsilon\) proves the result.
\end{proof}

\section{Continuity and dimension}

Katsumasa, Roff, and Yoshinaga formulated the following conjecture:
\begin{quote}
Let \(U\) be a finite-dimensional, positive definite normed
space. Then the function \(X\mapsto |X|\) is continuous, with
respect to the Hausdorff topology, on the set of finite subsets
of \(U\). In particular, magnitude is stable for finite subsets
of Euclidean or taxicab space.
\end{quote}
See \cite[Conjecture~1.3]{KatsumasaRoffYoshinaga}.
Kališnik and Lešnik subsequently proved continuity at the skew
finite subsets of \(\ell_1^N\), which form an open dense class
\cite{KalisnikLesnikSkew}. The following theorem proves the
conjecture for every finite subset of every finite-dimensional
positive definite normed space.

\begin{theorem}
\label{thm:finite-subset-continuity}
Let \(U\) be a finite-dimensional positive definite normed space.
Suppose that \(F,X_1,X_2,\ldots\) are nonempty finite subsets of
\(U\) and
\[
        d_H(X_j,F)\longrightarrow 0.
\]
Then
\[
        |X_j|\longrightarrow |F|.
\]
\end{theorem}

\begin{proof}
Write
\[
        F=\{p_1,\ldots,p_k\},
        \qquad
        \eta_j=d_H(X_j,F).
\]
Set
\[
 \delta_F=
 \begin{cases}
  \min_{i\neq\ell}\|p_i-p_\ell\|,&k\geq2,\\
  +\infty,&k=1.
 \end{cases}
\]
For all sufficiently large \(j\), one has
\(2\eta_j<\delta_F\).  The definition of Hausdorff distance then
shows that the sets
\[
        A_{j,i}
        =X_j\cap\overline B(p_i,\eta_j),
        \qquad 1\leq i\leq k,
\]
are nonempty and cover \(X_j\), while the inequality
\(2\eta_j<\delta_F\) makes them pairwise disjoint.  They therefore
form a partition of \(X_j\).

We first choose a convenient Hilbert-space realization of \(U\).
By the Blaschke--Lévy representation
\cite[Proposition~2.2]{MeckesHolmesThompson}, there is a finite
even Borel measure \(\sigma\) on \(S_{U^*}\) such that
\[
        \|u\|
        =\int_{S_{U^*}}|\theta(u)|\,d\sigma(\theta)
        \qquad (u\in U).
\]
For \(a\in\mathbb R\), let
\[
g_a(s)=
\begin{cases}
 \mathbf 1_{[0,a]}(s),&a\geq0,\\
-\mathbf 1_{[a,0]}(s),&a<0.
\end{cases}
\]
In
\[
        H=L_2(S_{U^*}\times\mathbb R,\sigma\otimes ds)
\]
define
\[
        (Jx)(\theta,s)=g_{\theta(x)}(s).
\]
Since
\[
        \int_{\mathbb R}|g_a-g_b|^2\,ds=|a-b|,
\]
we have
\[
        \|Jx-Jy\|_H^2=\|x-y\|.
\]

We shall need the fact that the closed unit ball \(B_U\) has finite
maximal energy. Define
\[
        c(\theta,s)
        =\frac12\operatorname{sgn}(s)\mathbf 1_{[-1,1]}(s).
\]
If \(x\in B_U\), then \(|\theta(x)|\leq1\), and consequently
\[
        |g_{\theta(x)}(s)-c(\theta,s)|=\frac12
\]
for almost every \(s\in[-1,1]\), while both functions vanish
outside \([-1,1]\). Hence
\[
        \|Jx-c\|_H^2
        =\frac12\sigma(S_{U^*})
        \qquad(x\in B_U).
\]
Thus \(x\mapsto Jx-c\) is a spherical Schoenberg realization of
\(B_U\). The spherical energy bound therefore gives
\[
        M(B_U)\leq \sigma(S_{U^*})<\infty.
\]

We now use the normalized coherent vectors
\[
 \varepsilon(h)
 =e^{-\|h\|^2}
   \bigoplus_{m=0}^{\infty}
   \frac{(\sqrt2\,h)^{\otimes m}}{\sqrt{m!}}
        \in\mathcal F_s(H).
\]
They satisfy
\[
        \langle\varepsilon(h),\varepsilon(k)\rangle
        =e^{-\|h-k\|^2}.
\]
Therefore
\[
        \Phi(x)=\varepsilon(Jx)
\]
is a feature map for the similarity kernel:
\[
        \langle\Phi(x),\Phi(y)\rangle=e^{-\|x-y\|}.
\]
For every finite \(E\subset U\), its magnitude is the squared norm
of the least-norm vector \(q\) satisfying
\[
        \langle q,\Phi(x)\rangle=1
        \qquad(x\in E).
\]

Let \(R_j\subset S_{U^*}\times\mathbb R\) be the union, over
\(1\leq i\leq k\) and \(x\in A_{j,i}\), of the intervals between
\(\theta(p_i)\) and \(\theta(x)\):
\[
 R_j=
 \bigcup_{i=1}^k\ \bigcup_{x\in A_{j,i}}
 \left\{(\theta,s):
 \min\{\theta(p_i),\theta(x)\}\leq s
 \leq\max\{\theta(p_i),\theta(x)\}\right\}.
\]
For each fixed \(\theta\), the part contributed by the \(i\)-th
cluster has length at most \(2\eta_j\). Consequently,
\[
        (\sigma\otimes ds)(R_j)
        \leq 2k\,\sigma(S_{U^*})\eta_j
        \longrightarrow0.
\]

The decomposition
\[
 H=L_2(R_j)\oplus L_2(R_j^c)
\]
induces the canonical factorization of symmetric Fock space. If
\(x\in A_{j,i}\), then \(Jx=Jp_i\) outside \(R_j\), and hence
\[
        \Phi(x)=a_{j,x}\otimes b_{j,i},
\]
where
\[
 a_{j,x}
 =\varepsilon\bigl((Jx)\mathbf1_{R_j}\bigr),
 \qquad
 b_{j,i}
 =\varepsilon\bigl((Jp_i)\mathbf1_{R_j^c}\bigr).
\]
Since \(|Jx|\leq1\) pointwise,
\[
 \bigl\|(Jx)\mathbf1_{R_j}\bigr\|_H^2
 \leq(\sigma\otimes ds)(R_j)\longrightarrow0
\]
uniformly in \(x\in X_j\). Thus \(a_{j,x}\) converges uniformly to
the vacuum vector \(\Omega\).

For each \(i\), let \(h_{j,i}\) be the least-norm vector satisfying
\[
        \langle h_{j,i},a_{j,x}\rangle=1
        \qquad(x\in A_{j,i}).
\]
All differences \(Jx-Jy\), with \(x,y\in A_{j,i}\), are supported
in \(R_j\). The Gram matrix of the vectors \(a_{j,x}\) is therefore
exactly the similarity matrix of \(A_{j,i}\).  This matrix is
strictly positive definite because \(U\) is positive definite.
Consequently the interpolation constraints are consistent and have a
unique least-norm solution in the span of the vectors \(a_{j,x}\), and
\[
        \|h_{j,i}\|^2=|A_{j,i}|.
\]
Since \(A_{j,i}\subset p_i+\eta_jB_U\), monotonicity, the
exponential magnitude bound, and the homogeneity of maximal
energy give
\[
  1\leq |A_{j,i}|
  \leq |\eta_jB_U|
  \leq e^{\eta_jM(B_U)}
  =1+o(1).
\]
This estimate is uniform in \(i\). For \(x\in A_{j,i}\),
\[
 \|h_{j,i}-a_{j,x}\|^2
 =\|h_{j,i}\|^2+\|a_{j,x}\|^2
   -2\langle h_{j,i},a_{j,x}\rangle
 =|A_{j,i}|-1=o(1).
\]
It follows that
\[
        h_{j,i}\longrightarrow\Omega
        \quad\text{and}\quad
        \langle h_{j,i},h_{j,\ell}\rangle\longrightarrow1
\]
for every \(i,\ell\).

Let
\[
        B_j=\bigl(\langle b_{j,i},b_{j,\ell}\rangle\bigr)_
             {i,\ell=1}^k.
\]
Its entries have the explicit form
\[
 (B_j)_{i\ell}
 =
 \exp\left(
 -\left\|
 (Jp_i-Jp_\ell)\mathbf1_{R_j^c}
 \right\|_H^2
 \right).
\]
The difference \(Jp_i-Jp_\ell\) has pointwise absolute value at
most \(1\), and therefore
\[
 \left\|
 (Jp_i-Jp_\ell)\mathbf1_{R_j}
 \right\|_H^2
 \leq(\sigma\otimes ds)(R_j)\longrightarrow0.
\]
Since \(\|Jp_i-Jp_\ell\|_H^2=\|p_i-p_\ell\|\), we obtain
\[
        (B_j)_{i\ell}
        \longrightarrow e^{-\|p_i-p_\ell\|}.
\]
Thus
\[
        B_j\longrightarrow
        Z_F=\bigl(e^{-\|p_i-p_\ell\|}\bigr)_{i,\ell=1}^k.
\]
The matrix \(Z_F\) is positive definite, so \(B_j\) is invertible
for all sufficiently large \(j\). Let \(d_{j,1},\ldots,d_{j,k}\)
be the basis dual to \(b_{j,1},\ldots,b_{j,k}\). Then
\[
        \langle d_{j,i},d_{j,\ell}\rangle
        =(B_j^{-1})_{i\ell}.
\]

Define
\[
        q_j=\sum_{i=1}^k h_{j,i}\otimes d_{j,i}.
\]
If \(x\in A_{j,\ell}\), then
\[
 \langle q_j,\Phi(x)\rangle
 =
 \sum_{i=1}^k
 \langle h_{j,i},a_{j,x}\rangle
 \langle d_{j,i},b_{j,\ell}\rangle
 =
 \langle h_{j,\ell},a_{j,x}\rangle
 =1.
\]
Thus \(q_j\) is an interpolant for \(X_j\), and hence
\[
\begin{aligned}
 \limsup_{j\to\infty}|X_j|
 &\leq\lim_{j\to\infty}\|q_j\|^2\\
 &=\lim_{j\to\infty}
   \sum_{i,\ell=1}^k
   \langle h_{j,i},h_{j,\ell}\rangle
   (B_j^{-1})_{i\ell}\\
 &=\mathbf1^{T}Z_F^{-1}\mathbf1
 =|F|.
\end{aligned}
\]

Finally, choose \(x_{j,i}\in X_j\) with
\(\|x_{j,i}-p_i\|\leq\eta_j\), and put
\[
        F_j=\{x_{j,1},\ldots,x_{j,k}\}.
\]
For large \(j\), these points are distinct. Their similarity
matrices converge to \(Z_F\), so
\[
        |F_j|\longrightarrow|F|.
\]
Since \(F_j\subseteq X_j\), monotonicity gives
\[
        |F_j|\leq|X_j|.
\]
Consequently,
\[
        |F|
        \leq\liminf_{j\to\infty}|X_j|
        \leq\limsup_{j\to\infty}|X_j|
        \leq|F|,
\]
which proves the theorem.
\end{proof}

Leinster and Meckes ask
\cite[Section~5, Question~(2)]{LeinsterMeckes}:
\begin{quote}
Is magnitude continuous on the class of compact sets in a positive
definite normed (or \(p\)-normed) space? What if we assume the space is
finite-dimensional, or we restrict to geodesic sets, or convex sets?
\end{quote}
All three proposed restrictions fail simultaneously.

\begin{theorem}[Discontinuity on compact convex sets]
\label{thm:convex-hausdorff-discontinuity}
There are compact convex sets
\(A_\varepsilon,A\subset\ell_1^3\) such that
\[
        d_H(A_\varepsilon,A)\longrightarrow0
        \qquad(\varepsilon\downarrow0),
\]
but
\[
        \lim_{\varepsilon\downarrow0}|A_\varepsilon|>|A|.
\]
Consequently, magnitude is not Hausdorff-continuous on compact subsets
of finite-dimensional positive definite normed spaces, even when
restricted to compact convex, and hence compact geodesic, subsets.
\end{theorem}

\begin{proof}
Let
\[
 \Delta=\left\{x\in\mathbb R_+^3:
                    x_1+x_2+x_3=1\right\},
 \qquad
 A=\frac1{10}\Delta,
\]
with the metric induced from \(\ell_1^3\), and set
\[
 C=\left[-\frac12,\frac12\right]^3,
 \qquad
 A_\varepsilon=A+\varepsilon C.
\]
Every \(A_\varepsilon\) is a convex body, \(A\) is compact and convex,
and
\[
        d_H(A_\varepsilon,A)\leq\frac32\varepsilon.
\]

For the interval-indicator realization \(J\) of \(\ell_1^3\), put
\[
        c=\frac13(Je_1+Je_2+Je_3).
\]
If \(x\in\Delta\), then
\[
 \|Jx-c\|^2
 =\sum_{i=1}^3
   \left(\frac49x_i+\frac19(1-x_i)\right)
 =\frac23.
\]
Since \(c\in\operatorname{aff}J(\Delta)\), the spherical energy identity
gives
\[
        M(\Delta)=\frac43,
        \qquad
        M(A)=\frac2{15}.
\]
The exponential magnitude bound therefore implies
\[
        |A|\leq e^{2/15}.
\]
Moreover,
\[
 e^{2/15}
 <
 1+\frac2{15}
 +\frac12\left(\frac2{15}\right)^2
 +\frac{(2/15)^3}{6(1-2/15)}
 =\frac{10027}{8775}
 <\frac{923}{800}.
\]

For a compact convex \(K\subset\ell_1^3\), write
\[
        W(K)=\sum_{i=0}^3 2^{-i}V_i'(K),
\]
where \(V_i'(K)\) denotes its \(i\)-th \(\ell_1\)-intrinsic volume.
For the simplex \(\Delta\),
the coordinate-projection definition of the \(\ell_1\)-intrinsic
volumes gives the required values directly.  Each of its three
one-coordinate projections is the interval \([0,1]\), and each of its
three two-coordinate projections is the triangle
\[
 \{(x,y)\in\mathbb R_+^2:x+y\leq1\},
\]
which has area \(1/2\); moreover, \(\Delta\) has zero three-dimensional
volume.  Hence
\[
 \bigl(V_0'(\Delta),V_1'(\Delta),
       V_2'(\Delta),V_3'(\Delta)\bigr)
 =\left(1,3,\frac32,0\right).
\]
Homogeneity therefore gives
\[
        W(A)
        =1+\frac12\frac3{10}
             +\frac14\frac{3}{200}
        =\frac{923}{800}.
\]

Leinster and Meckes proved that
\[
        |K|=W(K)
\]
for every convex body \(K\subset\ell_1^3\)
\cite[Theorem~4.6(2)]{LeinsterMeckes}. Hence
\[
        |A_\varepsilon|=W(A_\varepsilon).
\]
The \(\ell_1\)-intrinsic volumes are Hausdorff-continuous, and so
\[
 \lim_{\varepsilon\downarrow0}|A_\varepsilon|
 =\lim_{\varepsilon\downarrow0}W(A_\varepsilon)
 =W(A)
 =\frac{923}{800}
 >|A|.
\]
\end{proof}

Leinster and Meckes also conjectured that
\[
        |K|=\sum_{i=0}^n2^{-i}V_i'(K)
\]
for every compact \(\ell_1\)-convex set
\cite[Conjecture~4.5]{LeinsterMeckes}. The same example disproves this
conjecture, including its compact-convex case.

\begin{corollary}[Failure of the \(\ell_1\)-intrinsic-volume formula]
\label{cor:l1-intrinsic-volume-counterexample}
The Leinster--Meckes intrinsic-volume conjecture is false for compact
convex subsets of \(\ell_1^3\).
\end{corollary}

\begin{proof}
The set \(A=\Delta/10\) is compact and convex, hence
\(\ell_1\)-convex, but the preceding computation gives
\[
        |A|
        \leq e^{2/15}
        <\frac{923}{800}
        =\sum_{i=0}^3 2^{-i}V_i'(A).
\]
Thus the conjectured identity fails for \(A\).
\end{proof}

Katsumasa, Roff, and Yoshinaga ask
\cite[Question~1.2]{KatsumasaRoffYoshinaga}:
\begin{quote}
Given a positive definite metric space \(W\), is magnitude continuous
with respect to the Hausdorff topology on the set of compact subspaces
of \(W\)?
\end{quote}

\begin{theorem}[Discontinuity in a fixed compact ambient space]
\label{thm:fixed-ambient-discontinuity}
There is a compact positive definite metric space \(W\) for which
magnitude is not Hausdorff-continuous on the set of compact subspaces
of \(W\). Thus Question~1.2 has a negative answer.
\end{theorem}

\begin{proof}
Let
\[
        W=[-1,1]^3\subset\ell_1^3.
\]
This is a compact positive definite metric space. Retain the compact
convex sets
\[
 A=\frac1{10}\Delta,
 \qquad
 A_\varepsilon
 =A+\varepsilon\left[-\frac12,\frac12\right]^3
\]
from Theorem~\ref{thm:convex-hausdorff-discontinuity}. For every
sufficiently small \(\varepsilon>0\), both \(A\) and \(A_\varepsilon\)
are compact subspaces of \(W\), and
\[
        d_H^W(A_\varepsilon,A)\longrightarrow0.
\]
Nevertheless, the preceding computation gives
\[
 \lim_{\varepsilon\downarrow0}|A_\varepsilon|
 =\frac{923}{800}
 >|A|.
\]
Hence magnitude is discontinuous on the compact subspaces of the single
compact space \(W\).
\end{proof}

For a metric space \(T\), let \(\operatorname{Fin}_+(T)\) and
\(\operatorname{Cmp}_+(T)\) denote its nonempty finite and nonempty
compact subspaces, respectively, equipped with the Hausdorff metric.
Recall that \(T\) is \emph{tractable} if it is positive definite, every
closed ball in \(T\) is compact, and every such ball has finite
magnitude.

Kališnik and Lešnik conjecture
\cite[Conjecture~5.1]{KalisnikLesnikTractable}:
\begin{quote}
If \(T\subseteq L_1\) is tractable, then magnitude is continuous on
both \(\operatorname{Fin}_+(T)\) and \(\operatorname{Cmp}_+(T)\).
\end{quote}
The compact-subspace assertion is false, even when \(T\) is bounded.

\begin{theorem}[Discontinuity in a bounded tractable space]
\label{thm:tractable-compact-discontinuity}
There is a bounded tractable metric subspace \(T\subseteq L_1\) such
that
\[
        K\longmapsto |K|,
        \qquad K\in\operatorname{Cmp}_+(T),
\]
is not Hausdorff-continuous. In particular, the compact-subspace
assertion of Kališnik--Lešnik Conjecture~5.1 is false.
\end{theorem}

\begin{proof}
Identify \(\ell_1^3\) isometrically with a finite-dimensional subspace
of \(L_1\), and let
\[
        T=[-1,1]^3\subset\ell_1^3.
\]
Thus \(T\) is a bounded closed metric subspace of \(\ell_1^3\).
Kališnik and Lešnik prove that every closed metric subspace of
\(\ell_1^3\) is tractable
\cite[Corollary~3.5]{KalisnikLesnikTractable}; hence \(T\) is
tractable.

Retain the compact convex sets
\[
 A=\frac1{10}\Delta,
 \qquad
 A_\varepsilon
 =A+\varepsilon\left[-\frac12,\frac12\right]^3.
\]
For every sufficiently small \(\varepsilon>0\), both sets lie in \(T\).
Since \(T\) carries the restricted \(\ell_1^3\)-metric,
\[
        d_H^T(A_\varepsilon,A)
        =d_H^{\ell_1^3}(A_\varepsilon,A)
        \longrightarrow0.
\]
Nevertheless, Theorem~\ref{thm:convex-hausdorff-discontinuity} gives
\[
        \lim_{\varepsilon\downarrow0}|A_\varepsilon|
        =\frac{923}{800}
        >|A|.
\]
Thus magnitude is not continuous on
\(\operatorname{Cmp}_+(T)\).
\end{proof}

Kališnik and Lešnik further conjecture
\cite[Conjecture~5.2]{KalisnikLesnikTractable} that, for every
tractable \(T\subseteq L_1\), magnitude on
\(\operatorname{Cmp}_+(T)\) is Lipschitz if and only if \(T\) is
bounded. The same space disproves the forward claim from boundedness.

\begin{corollary}[Failure of bounded tractable Lipschitz continuity]
\label{cor:tractable-lipschitz-counterexample}
The implication
\[
 T\ \text{bounded}
 \quad\Longrightarrow\quad
 \bigl(K\mapsto|K|\bigr)
 \ \text{is Lipschitz on }\operatorname{Cmp}_+(T)
\]
in Kališnik--Lešnik Conjecture~5.2 is false.
\end{corollary}

\begin{proof}
The tractable space \(T=[-1,1]^3\) is bounded, but
Theorem~\ref{thm:tractable-compact-discontinuity} shows that magnitude
is not continuous on \(\operatorname{Cmp}_+(T)\). It therefore cannot
be Lipschitz there.
\end{proof}

For bi-Lipschitz equivalent compact metric spaces \(P\) and \(Q\), write
\[
 d_{\mathrm{Lip}}(P,Q)
 =
 \inf_f
 \log\max\!\left\{
     \operatorname{Lip}(f),
     \operatorname{Lip}(f^{-1})
 \right\},
\]
where the infimum is over bi-Lipschitz bijections \(f:P\to Q\).
Its bounded version is obtained by truncating this metric at \(1\);
the two metrics have the same convergent sequences.

The Magnitude 2019 problem list records the following question of
Meckes \cite[Question~1]{Magnitude2019Problems}:
\begin{quote}
Is magnitude continuous with respect to the bounded Lipschitz distance
on the collection of metric spaces of negative type homeomorphic to a
given space?
\end{quote}
The following gives a negative answer, even under simultaneous
Hausdorff convergence inside one compact ambient space.
\begin{theorem}[Discontinuity in the Lipschitz topology]
\label{thm:bilipschitz-magnitude-discontinuity}
There are compact spaces \(X,X_1,X_2,\ldots\) of negative type, all
contained in one compact space \(W\) of negative type, such that
\[
 d_{\mathrm{Lip}}(X_n,X)\longrightarrow0,
 \qquad
 d_H^W(X_n,X)\longrightarrow0,
\]
but
\[
 |X|<\infty,
 \qquad
 |X_n|\longrightarrow\infty.
\]
In particular, every \(X_n\) is homeomorphic to \(X\).
\end{theorem}

\begin{proof}
We begin with the six-point metric space due to Willerton and recorded
by Leinster \cite[Example~2.2.8, p.~876]{Leinster}.  Let
\[
 Y=A_0\sqcup B_0,
 \qquad
 A_0=\{a_1,a_2,a_3\},
 \qquad
 B_0=\{b_1,b_2,b_3\},
\]
and give distinct points of \(A_0\) distance \(1\), distinct points of
\(B_0\) distance \(2\), and points in different blocks distance \(1\).

We first verify the negative-type property needed below.  Let
\((u,v)\in\mathbb R^3\oplus\mathbb R^3\) have total mass zero, and put
\[
 a=\mathbf1^Tu=-\mathbf1^Tv.
\]
If \(D_Y\) denotes the distance matrix of \(Y\), then
\[
 (u,v)^TD_Y(u,v)
 =a^2-\|u\|^2-2\|v\|^2.
\]
By the Cauchy--Schwarz inequality,
\[
 \|u\|^2\geq\frac{a^2}{3},
 \qquad
 \|v\|^2\geq\frac{a^2}{3}.
\]
It follows that
\[
 (u,v)^TD_Y(u,v)
 \leq
 a^2-\frac{a^2}{3}-\frac{2a^2}{3}
 =0,
\]
so \(Y\) is of negative type.

For completeness, we also recall the magnitude calculation.  Put
\(q=e^{-t}\).  By symmetry, the weighting on \(tY\) is constant on
each of the two blocks; if its values are \(\alpha\) on \(A_0\) and
\(\beta\) on \(B_0\), then
\[
 \begin{pmatrix}
  1+2q & 3q\\
  3q   & 1+2q^2
 \end{pmatrix}
 \begin{pmatrix}\alpha\\ \beta\end{pmatrix}
 =
 \begin{pmatrix}1\\1\end{pmatrix}.
\]
Since
\[
 (1+2q)(1+2q^2)-9q^2
 =(1-q)^2(1+4q),
\]
solving these equations gives
\begin{equation}
\label{eq:six-point-magnitude}
 |tY|=3(\alpha+\beta)
     =\frac{6}{1+4e^{-t}},
 \qquad
 \lim_{t\downarrow0}|tY|=\frac65.
\end{equation}

For each \(k\geq1\), let \(Y_k\) be obtained from \(Y\) by adding
\(1/k\) to every nonzero distance:
\[
 d_{Y_k}(x,y)
 =
 d_Y(x,y)+\frac1k\mathbf1_{x\neq y}.
\]
If \(h\neq0\) has total mass zero, then
\[
 h^TD_{Y_k}h
 =
 h^TD_Yh-\frac1k\|h\|^2<0.
\]
Thus \(Y_k\) has strict negative type and consequently has the
one-point property:
\[
 \lim_{t\downarrow0}|tY_k|=1.
\]

Choose \(0<a_k<2^{-k}\) sufficiently small that
\begin{equation}
\label{eq:branch-magnitude-estimates}
 0\leq |a_kY_k|-1<2^{-k-1},
 \qquad
 \left||a_kY|-\frac65\right|<2^{-k-1}.
\end{equation}
Point \(Y\) and every \(Y_k\) at the same chosen point, and form the
pointed wedge
\[
 X=\bigvee_{k\geq1}a_kY_k.
\]
Each branch is finite, and
\[
 \operatorname{diam}(a_kY_k)
 \leq 3a_k<3\cdot2^{-k}.
\]
Given any sequence in \(X\), either some branch contains infinitely
many of its terms, in which case it has a constant subsequence, or
there is a subsequence whose branch indices tend to infinity, in
which case that subsequence converges to the common basepoint.
Therefore \(X\) is compact.  The orthogonal-sum construction for
pointed Schoenberg realizations shows that \(X\) is of negative type.

Every finite subset of \(X\) is contained in a finite subwedge.
Hence, by the definition of magnitude for compact spaces of negative
type, inclusion monotonicity, and the finite wedge formula,
\begin{align*}
 |X|
 &=
 \sup_{N\geq1}
 \left|\bigvee_{k=1}^N a_kY_k\right| \\
 &=
 1+\sum_{k\geq1}\bigl(|a_kY_k|-1\bigr)
 <\frac32.
\end{align*}
In particular, \(X\) has finite magnitude.

For \(n\geq1\), put
\[
 I_n=\{n,n+1,\ldots,2n-1\},
\]
and replace the branches indexed by \(I_n\) with copies of \(a_kY\):
\[
 X_n=
 \left(\bigvee_{k\notin I_n}a_kY_k\right)
 \vee
 \left(\bigvee_{k\in I_n}a_kY\right).
\]
The branch diameters again tend to zero, so the preceding argument
shows that every \(X_n\) is compact and of negative type.  The
finite-subwedge argument and \eqref{eq:branch-magnitude-estimates}
give
\begin{align*}
 |X_n|-|X|
 &=
 \sum_{k=n}^{2n-1}
 \bigl(|a_kY|-|a_kY_k|\bigr) \\
 &>
 \sum_{k=n}^{2n-1}
 \left(\frac15-2^{-k}\right) \\
 &>
 \frac n5-2^{1-n}.
\end{align*}
Consequently,
\[
 |X_n|\longrightarrow\infty.
\]

Let \(f_n:X\to X_n\) be the natural branchwise bijection.  It is the
identity on every unchanged branch and identifies \(Y_k\) with \(Y\)
on each branch indexed by \(I_n\).  On such a branch,
\[
 d_Y(x,y)
 \leq d_{Y_k}(x,y)
 =d_Y(x,y)+\frac1k
 \leq\left(1+\frac1k\right)d_Y(x,y),
\]
because every nonzero distance in \(Y\) is at least \(1\).
Distances between different wedge branches are sums of distances to
the common basepoint, so the same estimates hold globally.  Hence
\[
 \operatorname{Lip}(f_n)\leq1,
 \qquad
 \operatorname{Lip}(f_n^{-1})\leq1+\frac1n.
\]
In fact \(\operatorname{Lip}(f_n)=1\), since unchanged branches
remain.  It follows that
\[
 d_{\mathrm{Lip}}(X_n,X)
 \leq
 \log\left(1+\frac1n\right)
 \longrightarrow0.
\]
The same maps \(f_n\) show that \(X_n\) is homeomorphic to \(X\).

It remains to place all these spaces in a single compact space of
negative type.  Let
\[
 E=\{e_y:y\in Y\}
\]
carry the six-point discrete metric, and let \(e_\ast\) correspond to
the chosen basepoint \(\ast\in Y\).  Define
\[
 V_k=Y\times_1\frac1kE
\]
and the pointed subsets
\[
 D_k=\{(y,e_y):y\in Y\},
 \qquad
 C_k=\{(y,e_\ast):y\in Y\}.
\]
For \(y,y'\in Y\),
\[
 d_{V_k}\bigl((y,e_y),(y',e_{y'})\bigr)
 =
 d_Y(y,y')+\frac1k\mathbf1_{y\neq y'}
 =
 d_{Y_k}(y,y'),
\]
whereas
\[
 d_{V_k}\bigl((y,e_\ast),(y',e_\ast)\bigr)
 =
 d_Y(y,y').
\]
Thus \(D_k\) and \(C_k\) are pointed isometric copies of \(Y_k\) and
\(Y\), respectively.  Pairing \((y,e_y)\) with \((y,e_\ast)\) also
gives
\[
 d_H^{V_k}(D_k,C_k)\leq\frac1k.
\]

Finally, form the pointed wedge
\[
 W=\bigvee_{k\geq1}a_kV_k.
\]
Each \(V_k\) is of negative type because its metric is the sum of two
negative-type metrics.  Placing pointed Schoenberg realizations of
the branches in mutually orthogonal Hilbert summands therefore shows
that \(W\) is of negative type.

The diameters of the spaces \(V_k\) are uniformly bounded, while
\(a_k<2^{-k}\); hence the diameters of the branches \(a_kV_k\) tend
to zero.  Given any sequence in \(W\), either some branch contains
infinitely many of its terms, yielding a convergent subsequence since
that branch is finite, or there is a subsequence whose branch indices
tend to infinity.  In the latter case the subsequence converges to
the common basepoint.  Thus \(W\) is sequentially compact and hence
compact.

Inside \(W\), the space \(X\) is the wedge of the subsets \(a_kD_k\).
The space \(X_n\) is obtained by replacing \(a_kD_k\) by \(a_kC_k\)
for \(k\in I_n\).  Consequently,
\[
 d_H^W(X_n,X)
 \leq
 \max_{k\in I_n}\frac{a_k}{k}
 \leq
 \frac{2^{-n}}{n}
 \longrightarrow0.
\]
This completes the proof.
\end{proof}

For a compact metric space \(K\), let \(N(K,\rho)\) be the least
number of closed \(\rho\)-balls required to cover \(K\). Its upper and
lower Minkowski dimensions are
\[
 \overline{\dim}_{\mathrm{Mink}}K
 =
 \limsup_{\rho\downarrow0}
 \frac{\log N(K,\rho)}{\log(1/\rho)},
 \qquad
 \underline{\dim}_{\mathrm{Mink}}K
 =
 \liminf_{\rho\downarrow0}
 \frac{\log N(K,\rho)}{\log(1/\rho)}.
\]
For a compact space \(K\) of negative type whose magnitude is finite
at every scale, define
\[
 \overline{\dim}_{\mathrm{Mag}}K
 =
 \limsup_{t\to\infty}
 \frac{\log|tK|}{\log t},
 \qquad
 \underline{\dim}_{\mathrm{Mag}}K
 =
 \liminf_{t\to\infty}
 \frac{\log|tK|}{\log t}.
\]

Meckes asks whether magnitude dimension agrees with Minkowski
dimension for every compact metric space of negative type
\cite[Section~7, after Corollary~7.4]{MeckesCapacity}.
Leinster and Meckes record the same question in
\cite[Section~5, Question~(6)]{LeinsterMeckes}.
The answer is negative in the strongest possible sense.

\begin{theorem}[Magnitude dimension need not equal Minkowski dimension]
\label{thm:infinite-magnitude-dimension}
There is a countable compact metric space \(X\) of negative type such
that
\[
 |tX|<\infty
 \qquad(t>0),
\]
but
\[
 \dim_{\mathrm{Mink}}X=0,
 \qquad
 \underline{\dim}_{\mathrm{Mag}}X
 =
 \overline{\dim}_{\mathrm{Mag}}X
 =+\infty.
\]
\end{theorem}

\begin{proof}
Retain the six-point negative-type space \(Y=A_0\sqcup B_0\) from
Theorem~\ref{thm:bilipschitz-magnitude-discontinuity}. For
\(\varepsilon>0\), let \(Y_\varepsilon\) be obtained by adding
\(\varepsilon\) to every nonzero distance in \(Y\), and put
\[
 M_\varepsilon(s)=|sY_\varepsilon|
 \quad(s>0),
 \qquad
 M_\varepsilon(0)=1,
\]
where the value at \(0\) is defined by continuous extension. The space
\(Y_\varepsilon\) has strict negative type.

Set
\[
 q=e^{-s(1+\varepsilon)},
 \qquad
 r=e^{-s(2+\varepsilon)}.
\]
Solving the two orbit equations for the weighting gives
\begin{equation}
\label{eq:perturbed-six-point-magnitude}
 M_\varepsilon(s)
 =
 \frac{N_\varepsilon(s)}{D_\varepsilon(s)},
 \qquad
 N_\varepsilon(s)=6(1+r-2q),
 \qquad
 D_\varepsilon(s)=(1+2q)(1+2r)-9q^2.
\end{equation}
Fix \(L>0\).  With \(s=\lambda\varepsilon\), Taylor's formula with
remainder, uniformly for \(0\leq\lambda\leq L\), gives
\[
 N_\varepsilon(\lambda\varepsilon)
 =\lambda\varepsilon^2
  \bigl(6(1+\lambda)+O_L(\varepsilon)\bigr),
\]
and
\[
 D_\varepsilon(\lambda\varepsilon)
 =\lambda\varepsilon^2
  \bigl(6+5\lambda+O_L(\varepsilon)\bigr).
\]
The factored expressions are understood by continuous extension at
\(\lambda=0\).  Since \(6+5\lambda\geq6\), division now shows that
\begin{equation}
\label{eq:critical-product-limit}
 M_\varepsilon(\lambda\varepsilon)
 \longrightarrow
 Q(\lambda)
 :=
 \frac{6(1+\lambda)}{6+5\lambda}
 \qquad(\varepsilon\downarrow0),
\end{equation}
uniformly for \(\lambda\) in compact subsets of \([0,\infty)\).
Moreover,
\[
 Q(\lambda)-1
 =
 \frac{\lambda}{6+5\lambda}
 \leq\frac{\lambda}{6}.
\]

We next establish the uniform lower bound
\begin{equation}
\label{eq:uniform-product-lower-bound}
 M_\varepsilon(s)\geq\frac{13}{12}
 \qquad
 (0<\varepsilon\leq\varepsilon_0,\ s\geq\varepsilon)
\end{equation}
for some \(\varepsilon_0>0\). We divide the argument into three
ranges.

First suppose that
\[
 \varepsilon\leq s\leq\delta
\]
and write \(u=\varepsilon/s\in[0,1]\).  Taylor's formula, now with
\(\varepsilon=us\), gives
\[
 N_{us}(s)=s^2\bigl(6(1+u)+O(s)\bigr),
 \qquad
 D_{us}(s)=s^2\bigl(5+6u+O(s)\bigr),
\]
uniformly for \(u\in[0,1]\).  Since \(5+6u\geq5\), it follows
uniformly that
\[
 M_{us}(s)
 =
 \frac{6(1+u)}{5+6u}+O(s).
\]
Since
\[
 \min_{0\leq u\leq1}
 \frac{6(1+u)}{5+6u}
 =
 \frac{12}{11}
 >
 \frac{13}{12},
\]
we may choose \(\delta>0\) so small that
\[
 M_\varepsilon(s)\geq\frac{13}{12}
 \qquad
 (\varepsilon\leq s\leq\delta).
\]

Next fix \(R>\delta\). On the compact interval \([\delta,R]\),
formula \eqref{eq:perturbed-six-point-magnitude} converges uniformly,
as \(\varepsilon\downarrow0\), to
\[
 M_0(s)=|sY|=\frac{6}{1+4e^{-s}}.
\]
Since \(M_0(s)\geq6/5>13/12\), there is
\(\varepsilon_1>0\) such that
\[
 M_\varepsilon(s)\geq\frac{13}{12}
 \qquad
 (0<\varepsilon\leq\varepsilon_1,\ \delta\leq s\leq R).
\]

Finally, for \(s\geq R\) one has
\[
 0\leq q\leq e^{-s}\leq e^{-R},
 \qquad
 0\leq r\leq e^{-2s}\leq e^{-2R}.
\]
It follows directly from
\eqref{eq:perturbed-six-point-magnitude} that
\(M_\varepsilon(s)\to6\) as \(R\to\infty\), uniformly in
\(\varepsilon>0\) and \(s\geq R\). Taking \(R\) sufficiently large,
we obtain
\[
 M_\varepsilon(s)\geq\frac{13}{12}
 \qquad(\varepsilon>0,\ s\geq R).
\]
After decreasing \(\varepsilon_1\), if necessary, so that
\(\varepsilon_1\leq\delta\), the choice
\(\varepsilon_0=\varepsilon_1\) proves
\eqref{eq:uniform-product-lower-bound}.

Choose
\[
 m_k=\left\lceil e^{k^2}\right\rceil,
 \qquad
 T_0=1,
 \qquad
 T_k=T_{k-1}m_k^2.
\]
Since \(Q(\lambda)\leq1+\lambda/6\),
\[
 \sup_{0\leq\lambda\leq m_k^{-2}}
 Q(\lambda)^{m_k}
 \leq
 \left(1+\frac{1}{6m_k^2}\right)^{m_k}
 \leq e^{1/(6m_k)}
 <1+2^{-k}.
\]
By the uniform convergence in
\eqref{eq:critical-product-limit}, for each \(k\) there is
\(\delta_k>0\) such that
\begin{equation}
\label{eq:product-tail-threshold}
 \sup_{0\leq\lambda\leq m_k^{-2}}
 M_\varepsilon(\lambda\varepsilon)^{m_k}
 <1+2^{-k}
 \qquad(0<\varepsilon<\delta_k).
\end{equation}
Thus the inequality remains valid whenever \(\varepsilon\) is
decreased while kept below \(\delta_k\).

For \(0<\varepsilon<\min\{\varepsilon_0,\delta_k\}\), let
\[
 P_k(\varepsilon)=Y_\varepsilon^{\times_1m_k},
 \qquad
 B_k(\varepsilon)=\frac{\varepsilon}{T_k}P_k(\varepsilon),
\]
pointed at the corresponding product basepoint, and let
\(r_k(\varepsilon)\) denote the radius of \(B_k(\varepsilon)\) about
that point. Since
\[
 r_k(\varepsilon)
 \leq
 \frac{\varepsilon}{T_k}m_k(2+\varepsilon),
\]
we have \(r_k(\varepsilon)\to0\) as \(\varepsilon\downarrow0\).

Set \(r_0=1\). We may therefore choose \(\varepsilon_k\) recursively
so that
\[
 0<\varepsilon_k
 <
 \min\{\varepsilon_0,\delta_k,\tfrac12\varepsilon_{k-1}\},
\]
and, on writing
\[
 P_k=P_k(\varepsilon_k),
 \qquad
 B_k=B_k(\varepsilon_k),
 \qquad
 r_k=r_k(\varepsilon_k),
\]
so that \(r_k<r_{k-1}\) and
\begin{equation}
\label{eq:branch-radius-condition}
 \log(1/r_k)
 \geq
 k^2
 \log\left(
 1+\sum_{j\leq k}(6^{m_j}-1)
 \right).
\end{equation}
By \eqref{eq:product-tail-threshold}, these choices also satisfy
\begin{equation}
\label{eq:product-tail-bound}
 \sup_{0\leq\lambda\leq m_k^{-2}}
 M_{\varepsilon_k}(\lambda\varepsilon_k)^{m_k}
 <1+2^{-k}.
\end{equation}

Now let
\[
 X=\bigvee_{k\geq1}B_k.
\]
Products and wedges preserve negative type, so \(X\) is of negative
type. Each branch is finite, with
\[
 \#B_k=6^{m_k},
\]
and hence \(X\) is countable. It is also compact. Indeed, given a
sequence in \(X\), either one branch contains infinitely many of its
terms, in which case it has a constant subsequence, or there is a
subsequence whose branch indices tend to infinity. The latter
subsequence converges to the common basepoint because \(r_k\to0\).

We next prove that \(X\) has finite magnitude at every scale. By
multiplicativity for \(\ell_1\)-products,
\begin{equation}
\label{eq:scaled-branch-magnitude}
 |tB_k|
 =
 M_{\varepsilon_k}
 \left(\frac{t\varepsilon_k}{T_k}\right)^{m_k}.
\end{equation}
Fix \(t>0\). For every sufficiently large \(k\), \(t\leq T_{k-1}\);
hence
\[
 0\leq\frac{t}{T_k}\leq\frac1{m_k^2}.
\]
Equations \eqref{eq:product-tail-bound} and
\eqref{eq:scaled-branch-magnitude} then give
\[
 1\leq |tB_k|<1+2^{-k}.
\]
Consequently
\[
 \sum_{k\geq1}\bigl(|tB_k|-1\bigr)<\infty.
\]

For \(N\geq1\), let
\[
 X_N=\bigvee_{k=1}^{N}B_k.
\]
The finite wedge formula gives
\[
 |tX_N|
 =
 1+\sum_{k=1}^{N}\bigl(|tB_k|-1\bigr).
\]
Since \(X_N\subseteq X\), inclusion monotonicity gives
\[
 1+\sum_{k\geq1}\bigl(|tB_k|-1\bigr)
 \leq |tX|.
\]
Conversely, every finite subset \(F\subseteq X\) lies in some
\(X_N\), and hence
\[
 |tF|
 \leq |tX_N|
 \leq
 1+\sum_{k\geq1}\bigl(|tB_k|-1\bigr).
\]
Taking the supremum over all finite \(F\subseteq X\) proves
\begin{equation}
\label{eq:infinite-wedge-magnitude}
 |tX|
 =
 1+\sum_{k\geq1}\bigl(|tB_k|-1\bigr)
 <\infty.
\end{equation}

We now compute the Minkowski dimension. If
\[
 r_{k+1}\leq\rho<r_k,
\]
then the tail \(\bigvee_{j>k}B_j\) lies in the closed
\(\rho\)-ball about the basepoint. Covering each remaining point
separately gives
\[
 N(X,\rho)
 \leq
 1+\sum_{j\leq k}(6^{m_j}-1).
\]
Since \(\rho<r_k\), condition
\eqref{eq:branch-radius-condition} implies
\[
 \frac{\log N(X,\rho)}{\log(1/\rho)}
 \leq
 \frac{
 \log\left(1+\sum_{j\leq k}(6^{m_j}-1)\right)}
 {\log(1/r_k)}
 \leq\frac1{k^2}.
\]
It follows that
\[
 \dim_{\mathrm{Mink}}X=0.
\]

Finally, suppose that \(T_k\leq t<T_{k+1}\). Then
\[
 \frac{t\varepsilon_k}{T_k}\geq\varepsilon_k,
\]
so \eqref{eq:uniform-product-lower-bound},
\eqref{eq:scaled-branch-magnitude}, and inclusion monotonicity give
\[
 |tX|
 \geq|tB_k|
 \geq\left(\frac{13}{12}\right)^{m_k}.
\]
Moreover,
\[
 \log T_{k+1}
 =
 2\sum_{j=1}^{k+1}\log m_j
 =
 O(k^3),
 \qquad
 m_k\geq e^{k^2}.
\]
Therefore, whenever \(T_k\leq t<T_{k+1}\),
\[
 \frac{\log|tX|}{\log t}
 \geq
 \frac{m_k\log(13/12)}{\log T_{k+1}}
 \longrightarrow+\infty.
\]
Thus
\[
 \underline{\dim}_{\mathrm{Mag}}X
 =
 \overline{\dim}_{\mathrm{Mag}}X
 =+\infty,
\]
as required.
\end{proof}

\section{Acknowledgements}
I thank Professor Kiyonori Gomi and Mark W. Meckes for analysis of the fourteen-point counterexample described in Theorem 4.1, valuable communications about several of the results in here, and for suggestions greatly improving the formatting and structure of this article.

\section{AI Use Disclosure}
Earlier in the research process, GPT 5.6 Sol was used in surveying literature to find necessary background work to read.

\end{document}